\documentclass[a4paper,11pt,reqno]{amsart}

\usepackage[latin1]{inputenc}
\usepackage{graphicx}
\usepackage{amsmath}
\usepackage{mathrsfs}

\DeclareFontFamily{U}{calligra}{}
\DeclareFontShape{U}{calligra}{m}{n}{<->callig15}{}

\newcommand{\llangle}{\hbox{$\langle\kern-1.5pt\langle$}}
\newcommand{\rrangle}{\hbox{$\rangle\kern-1.5pt\rangle$}}

\usepackage{amsmath,stmaryrd, amsthm,latexsym,amscd,amssymb,MnSymbol} 
\usepackage{graphicx}

\usepackage[all]{xy}
\usepackage{a4}
\usepackage{hyperref}
\usepackage{url}
\usepackage{fancyhdr}
\usepackage{amscd}
\usepackage[usenames,dvipsnames]{xcolor}
\numberwithin{equation}{section}
\newtheorem{prop}{Proposition}[section]
\newtheorem{lemma}[prop]{Lemma}
\newtheorem{theorem}[prop]{Theorem}

\newtheorem{conjecture}[prop]{``Conjecture''}
\theoremstyle{definition}
\newtheorem{definition}[prop]{Definition}

\theoremstyle{remark}
\newtheorem{claim}[prop]{Claim}
\newtheorem{remark}[prop]{Remark}
\newtheorem{remarks}[prop]{Remarks}

\def\id{\mathop{\rm id}}
\def\tr{\mathop{\rm tr}}
\def\phi{\varphi}

\def\G{\Gamma}

\def\R{{\mathbin{{\mathbb {R}}}}}
\def\Z{\mathbin{\mathbb {Z}}}
\def\C{\mathbb{C}}

\begin{document}

\title[The Baum--Connes conjecture localised at the unit of a discrete group]{The Baum--Connes conjecture localised at the unit element of a discrete group}
\author{Paolo Antonini}
%\email{pantonin@sissa.it}
%\address{Paolo Antonini\\Scuola Internazionale Superiore di Studi Avanzati\newline
%via Bonomea, 265\newline
%34136 Trieste, Italy
%}

%
\author{Sara Azzali}
%\email{sara.azzali@uni-hamburg.de}
%\address{Sara Azzali\\ Fachbereich Mathematik\newline
%Universit\"at Hamburg\newline
%Bundesstrasse 55\newline
%20146 Hamburg, Germany}
%
%
\author{Georges Skandalis}
%\email{skandalis@math.univ-paris-diderot.fr}
%\address{Georges Skandalis\\
%Universit\'e Paris Diderot, Sorbonne Paris Cit\'e\newline
%Sorbonne Universit\'es, UPMC Paris 06, CNRS, IMJ-PRG\newline
%UFR de Math\'ematiques, {\sc CP} {\bf 7012} - B\^atiment Sophie Germain \newline
%5 rue Thomas Mann, 75205 Paris CEDEX 13, France}

%\keywords{assembly maps, Novikov conjecture, equivariant Kasparov theory, classifying spaces}
\thanks{
\smallskip
\smallskip
\newline Paolo Antonini was partially supported by
the grant H2020-MSCA-RISE-2015-691246-QUANTUM DYNAMICS;  Sara Azzali  acknowledges support from the DFG grant \emph{Secondary invariants for foliations} within the Priority Programme SPP 2026 \lq\lq Geometry at Infinity\rq\rq.}

\begin{abstract}
We construct a Baum--Connes assembly map localised at the unit element of a discrete group $\Gamma$. 
This morphism, called $\mu_\tau$, is defined in $KK$-theory with coefficients in $\mathbb{R}$ by means of the action of the idempotent $[\tau]\in KK_\R^\Gamma(\C,\C)$ canonically associated to the group trace of $\Gamma$. We show that the corresponding $\tau$-Baum--Connes conjecture is weaker then the classical one but still implies the strong Novikov conjecture. 
The right hand side of $\mu_\tau$ is functorial with respect to the group $\Gamma$.
\end{abstract}

\maketitle

\setcounter{tocdepth}{1}

\tableofcontents

\section{Introduction}
\label{sec:introduction}

The Baum--Connes conjecture \cite{BC} predicts an isomorphism between two abelian groups naturally associated with 
a discrete group $\Gamma$.
More precisely, an assembly map 
$$
\mu:K^{\textrm{top}}_*(\Gamma) \longrightarrow K_*(C^*_r\Gamma)
$$
is constructed and is conjectured to be an isomorphism. 
The left hand side is a topological object, based on the $K$-homology of proper $\Gamma$-spaces with compact quotient;  the right hand side is analytic, the $K$-theory of the reduced $C^*$-algebra of $\Gamma$.

We refer to the vast literature on the subject (see \cite{BC, BCH}, the book \cite{Va} and the recent articles \cite{BGW, BEW, GJV} and the references therein). 
One of the main motivations of the conjecture is that the injectivity of $\mu$ implies the Novikov conjecture.

In fact, the Novikov conjecture is implied by the rational injectivity of the Mishchenko--Kasparov assembly map \cite{Kas2,MishFo}
$$
\tilde \mu: K_*(B\Gamma) \longrightarrow K_*(C^*_r\Gamma)
$$
and the Baum--Connes map incorporates the Mishchenko--Kasparov assembly in the sense that $\tilde \mu=\mu\circ \sigma$, where $\sigma:K_*(B\Gamma)\to K^{\textrm{top}}_*(\Gamma)$ is a natural, rationally injective \lq\lq forgetful\rq\rq\ map. When $\Gamma$ is torsion free, one simply has $\mu=\tilde \mu$. In general, the left hand side $K^{\textrm{top}}_*(\Gamma)$  of the Baum--Connes map also contains information coming from finite order elements of $\Gamma$, \cite{BCH}.

The more general formulation of the Baum--Connes conjecture  is the one called  \lq
\lq with coefficients\rq\rq, which involves the action of a discrete group $\Gamma$ on a $C^*$-algebra $A$. In this case, the conjecture about the bijectivity of the map
$$
\mu^A:K^{\textrm{top}}_*(\Gamma;A) \longrightarrow K(A\rtimes_r\Gamma)
$$
is known to admit counterexamples \cite{HLS}.

\bigskip In the present paper we apply the model of $KK$-theory with real coefficients developed in \cite{AAS2} to study the properties of the assembly maps. We explain how  adding coefficients in $\R$ provides a way to localise at the unit element of $\Gamma$. 
The upshot is a localised form of the Baum--Connes conjecture 
that we relate to the classical Baum--Connes and Novikov conjectures.

In \cite{AAS2} we identified a distinguished idempotent element  $[\tau]$ of the commutative ring $KK_\R^\Gamma(\C,\C)$ which is canonically associated to the group $\Gamma$ via its group trace.
If the group $\Gamma$ acts freely and properly on a locally compact space $Y$, then  $[\tau]$ acts as the identity element on the $\Gamma$-algebra of functions $C_0(Y)$, \textit{i.e}.\ it defines the unit element of the ring $KK_{\R}^\Gamma(C_0(Y),C_0(Y))$.

Via the action of the idempotent $[\tau]$  
we define 
for $\Gamma$-algebras $A$, $B$ the $\tau$-part of $KK^\Gamma_\R(A, B)$ to be the subgroup
$$
 KK^\Gamma_\R(A, B)_{\tau}:=\{x\otimes [\tau]; \ x\in KK^\Gamma_\R(A, B)\}=\{x\in KK^\Gamma_\R(A, B)\;: x\otimes [\tau]=x\}
$$
on which $[\tau]$ acts as the identity.  

\medskip The map $\mu_\tau$ that we construct is defined on $KK$-theory with real coefficients and relates the $\tau$-parts of the left and right hand side of the Baum--Connes assembly map.

\begin{itemize}
\item The left hand side of the $\tau$-assembly morphism will be  $K^{\textrm{top}}_{*, \R}(\Gamma;A)_{\tau}$. 

\item For the right hand side, it is natural to define the $\tau$-part of $K_\R(A\rtimes_r \G)$ letting $[\tau]$ act via descent \textit{i.e}.\ by right multiplication with the idempotent element $J^\Gamma_r(1_A\otimes [\tau])$ of the ring $KK_\R(A\rtimes_r \G, A\rtimes_r \G)$. We thus set
$$
K_{*,\R}(A\rtimes_r \G)_{\tau}:=\{\xi\in K_{*,\R}(A\rtimes_r \G):\, \xi\otimes J^\Gamma_r(1_A\otimes [\tau])=\xi\}.
$$
\item The assembly morphism defines a map 
$$
\mu_\tau: K^{\textrm{top}}_{*, \R}(\Gamma;A)_{\tau}\longrightarrow K_\R(A\rtimes_r \G)_{\tau} \ .
$$

\end{itemize}
The $\tau$-form of the Baum--Connes conjecture with coefficients in a $\Gamma$-algebra $A$ is then the statement that $\mu_\tau$ is an isomorphism. 

\bigskip We show that the $\tau$-form of the Baum--Connes conjecture is weaker than the Baum--Connes conjecture, in the sense that,  if the Baum--Connes assembly map is injective (\emph{resp}.\ surjective) for $A\otimes N$ for every $\rm{II}_1$-factor $N$ then $\mu_\tau$ is injective (\emph{resp}.\ surjective) for $A$ (Theorem \ref{thm:BCImpliesBCtau}).

\bigskip Moreover, the  $\tau$-form of Baum--Connes is \lq\lq closer\rq\rq{} to the Novikov conjecture, since it is only concerned with the part of $K_*^{top}(\Gamma)$ corresponding to free and proper actions. We show 

\bigskip
\noindent\textbf{Theorem \ref{theorem:inj-novikov}.} \emph {If $\mu_\tau:  K^{\textrm{top}}_{*, \R}(\Gamma)_{\tau}\longrightarrow K_\R(C^*_r \G)_{\tau}$ is injective, then the Mishchenko--Kasparov assembly map $K_*(B\Gamma)\to K_*(C_r^*\Gamma)$ is rationally injective. } \\

To do so, exploiting the properties of $KK$-theory with real coefficients, we prove (Theorem \ref{th:sigmat}) that the $\tau$-part of $K^{\textrm{top}}_{*, \R}(\Gamma)$ identifies with $K_{*,\mathbb R}(B\Gamma)$ and, more generally,  the $\tau$-part of $K^{\textrm{top}}_{*, \R}(\Gamma;A)$ identifies with $K^\Gamma_{*,\mathbb R}(E\Gamma; A)$.

We already know from \cite{AAS2} that the idempotent $[\tau]$ acts as the identity on $K^\Gamma_{*,\mathbb R}(E\Gamma; A)$.

To show that  the $\tau$-part of $K^{\textrm{top}}_{*, \R}(\Gamma;A)$ is contained in  $K^\Gamma_{*,\mathbb R}(E\Gamma; A)$, we construct a compact $\Gamma$-space $X$ with  an invariant probability measure on which each torsion subgroup of $\Gamma$ acts freely.
In particular, this provides an inverse $t$ to the map $\sigma: K_{*,\mathbb R}(B\Gamma)\to K^{\textrm{top}}_{*, \R}(\Gamma)_\tau $ on the $\tau$-part.

The relation between the localised Baum--Connes assembly map  $\mu_\tau$ and the Mishchenko--Kasparov assembly $\tilde \mu$  is summarised by the following commutative diagram (Remark \ref{diagram}):
$$
	\xymatrix{&K^{\textrm{top}}_{*,\mathbb R}(\Gamma)_{\tau}\ar[r]^{\mu_\tau\;}\ar@{-->}@/^1.2pc/[dl]^t &K_{*,\mathbb R}(C_r^*\Gamma)_\tau\ar@{^{(}->}[d] \\
		K_{*,\mathbb R}( B\Gamma)\ar[ur]^{\sigma_*} 
&& K_{*,\mathbb R}(C_r^*\Gamma)\\
		K_{*}( B\Gamma)\otimes \mathbb R \ar[u]^{\cong}\ar[rr]_{\tilde \mu\otimes 1} && K_{*}(C_r^*\Gamma)\otimes \mathbb R\ar[u]_\beta
	}
$$

This also provides informations about the image of the assembly map $\tilde \mu\otimes 1$. More precisely, applying the natural map $\beta$, the image of $\tilde \mu\otimes 1$ is in the $\tau$-part of $K_{*,\R}(C_r^*\Gamma)$.

\bigskip

Let us also mention some other potentially important features of our construction. First, the localised map  $\mu_\tau$ does not distinguish between full and reduced group $C^*$-algebra: in fact we observe (Proposition \ref{prop:redmax}) that 
$$
K_{*,\R}(C^*\Gamma)_\tau\simeq K_{*,\R}(C_r^*\Gamma)_\tau\ .
$$
The second feature is related to functoriality.
It is known that the reduced $C^*$-algebra of groups is not functorial with respect to (non injective) morphisms $\Gamma_1\to \Gamma_2$ of countable groups, unlike the left hand side of $K_*^{\rm top}(\Gamma)$. On the other hand, we show that the group $K_{*,\R}(C^*\Gamma)_\tau$ is  functorial (see Section \ref{naturalitytau}).

\medskip 
From the above properties, arguing that 
the Baum--Connes conjecture holds for hyperbolic groups (\cite{LafforgueHyperbolic}), one could hope to be able to establish bijectivity of $\mu_\tau$ in a great generality. 

\medskip For the case with coefficients in a $\Gamma$-algebra this is unfortunately not the case: in Section \ref{sec:counterex} we show that the construction of  \cite{HLS} for group actions using the ``Gromov monster'' still provides counterexamples to the bijectivity of the localised Baum--Connes assembly $\mu_\tau:  K^{\textrm{top}}_{*, \R}(\Gamma;A)_{\tau}\longrightarrow K_\R(A\rtimes_r \G)_{\tau} $ for a suitable choice of coefficents $A$. Again the failure of exactness is the source of  counterexamples. 
 In fact, using the properties of a Gromov monster group $\Gamma$ (\cite{AD, Gr}), and letting $A:=\ell^{\infty}(\mathbb{N};c_0(\Gamma))$ we show that the sequence of the $\tau$-parts 
$$	 \xymatrix{ K_{0,\R}(c_0(\mathbb{N}\times \Gamma) \rtimes_r \Gamma )_{\tau}\ar[r]&  K_{0,\R}(A \rtimes_r \Gamma \ar[r])_{\tau}&   K_{0,\R}\big{(} \big{(}A/c_0(\mathbb{N}\times \Gamma) \big{)}\rtimes_r \Gamma    \big{)}_{\tau} }	 
 $$
is not exact in the middle. This proves that $\mu_\tau$ cannot be an isomorphism for every $\Gamma$-algebra. 

\bigskip
\textbf{On the terminology.} Let's say in what sense the element $\tau $ \emph{localises at the unit element.} \begin{itemize}
\item Naively, $\tau$ really takes the value of a function at the unit element.
\item In a more sophisticated sense, we know from \cite{BCH} that the left hand side $K^{\textrm{top}}_{*}(\Gamma)$ rationally breaks into contributions of the various conjugacy classes of $\Gamma$. In the same way, it is known from \cite{Bur} that the cyclic cohomology of the group algebra $\C\Gamma$ breaks into contributions of the various conjugacy classes of the group; the one associated to the unit element is equal to the cohomology group $H^*(\Gamma;\C)=H^*(B\Gamma;\C)$. The idempotent $[\tau]$ isolates this component at the level of the ``right hand side'' $K_*(C_r^*\Gamma)$.
\end{itemize}

\bigskip \textbf{Notation.} All the tensor products considered here are minimal tensor products and will be just written with the sign ``$\otimes$''. \\
We will use the sign ``$\otimes_D$'' for the Kasparov product of $KK$-elements over a $C^*$-algebra $D$ in the sense of \cite{Kas1, Kas2}. When $x\in KK(A,D)$ and $y\in KK(D,B)$ we will sometimes drop the subscript $D$ and write $x\otimes y$ instead of $x\otimes_D y$.

Throughout the paper, when dealing with classifying spaces $E\G$ and $B\G$, we always use the notation $K_*(B\Gamma)$ and or $K_*^\Gamma(E\Gamma)$ to mean the $K$-homology with compact supports (or $\Gamma$-compact supports), which are often 
denoted by $RK_*$ and $RK_*^\Gamma$ in the literature.  

\bigskip \textbf{Acknowledgements.} The authors wish to thank the referee for many helpful suggestions which greatly improved the paper.

\section{$KK$-theory with real coefficients}
\label{sec:realKK}
In this section we briefly recall some constructions we will need later.
\subsection{Non-separable $C^*$-algebras and $K$-theory}
\label{separability} 
Let us recall some material from \cite{Sk2}. One shows that, when $A$ is a separable $C^*$-algebra, then for every $C^*$-algebra $B$, the group $KK(A,B)$  is the inductive limit over all the separable subalgebras of $B$. Based on this,  one constructs a new group $KK_{sep}(A,B)$ for non necessarily separable $A$ and $B$;  it is defined as the projective limit of the groups $KK(A_1,B)$ where $A_1$ runs over all the separable subalgebras $A_1 \subset A$. It enjoys two important properties:
\begin{itemize}
\item the familiar Kasparov group maps to this new $KK(A,B)$,
\item in the new $KK(A,B)$ the Kasparov product $$\otimes_D:KK_{sep}(A_1,B_1\otimes D) \times KK_{sep}(D\otimes A_2,B_2) \longrightarrow KK_{sep}(A_1\otimes A_2, B_1 \otimes B_2) $$
is always defined without separability assumptions. This follows from the naturality of the Kasparov product in the separable case.
\end{itemize}
One may note that all these facts can immediately be extended to the equivariant case - with respect to a second countable group $\Gamma$.
We have

\begin{lemma}
\label{lem:sep} Let $\,\Gamma$ be a countable discrete group,  and $B$ a $\Gamma$-algebra. 
\begin{enumerate}
\item For every separable $\Gamma $ algebra $A$, the group $KK^\Gamma(A,B)$ is the inductive limit of $KK^\Gamma(A,B_1)$ where $B_1$ runs over separable  $\Gamma$-subalgebras $B_1 \subset B$.
\item For every separable algebra $A$, the group $KK(A;B\rtimes_r \Gamma)$ is the inductive limit of $KK(A,B_1\rtimes_r \Gamma)$ where $B_1$ runs over separable  $\Gamma$-subalgebras $B_1 \subset B$.
\end{enumerate}
\end{lemma}
The proof of (1) goes as in \cite[Rk. 3.2]{Sk2}, by using that $\Gamma$ is a countable group. 

Statement (2) follows from the fact that every separable subalgebra of $B\rtimes_r \Gamma$ is contained in a $B_1\rtimes_r \Gamma$.

\smallskip Note that the same remarks hold also for full crossed products.

\subsection{$KK$-theory with real coefficients}
In \cite{AAS2}, we constructed the functor $KK^G_{\R}$ by taking an inductive limit over $\rm{II}_1$-factors. More precisely, given a locally compact group(oid) and two $G$-algebras $A$ and $B$, we defined $KK^G_\R(A,B)$ as the limit of $KK^G(A,B\otimes M)$ where $M$ runs over $\rm{II}_1$-factors with trivial $G$ action and unital embeddings.

An equivalent definition of $KK^G_\R(A, B)$ can be given by taking  limits of $KK^G_\R(A, B\otimes D)$ over pairs $(D,\lambda)$ where $D$ is a unital $C^*$-algebra with trivial $G$-action, and $\lambda$ is a tracial state on $D$ \cite[Remarks 1.6]{AAS2}.

\subsection{The Kasparov product in $KK^G_{\R}$} Let $x_1\in KK^G_{\R}(A_1,B_1\otimes D)$ and $x_2\in KK^G_{\R}(D\otimes  A_2,B_2)$; represent them as elements $y_1\in KK^G_{\R}(A_1,B_1\otimes M_1\otimes D)$ and $y_2\in KK^G(D\otimes A_2,M_2\otimes B_2)$ where $M_1$ and $M_2$ are $\rm{II}_1$-factors. The Kasparov product $(y_1\otimes 1_{A_2})\otimes_D (1_{B_1\otimes M_1}\otimes y_2)$ is an element in $KK^G(A_1\otimes A_2,B_1\otimes M_1\otimes M_2\otimes B_2)$. Finally using the canonical morphism $M_1\otimes M_2\to M_1\overline {\otimes} M_2$ to the von Neumann tensor product, we obtain an element $x_1\otimes _Dx_2\in KK^G_{\R}(A_1\otimes A_2,B_1\otimes B_2)$.

The Kasparov product in $KK^G_{\R}$ has all the usual properties (bilinearity, associativity, functoriality, ...). In particular, we have:
\begin{lemma}
\label{commute}
The exterior product in $KK^G_{\R}$ is commutative.  
\end{lemma}
\begin{proof}Let  $A_1, A_2, B_1, B_2$ be $G$-algebras, $x_i\in KK_\R^G(A_i,B_i)$. There exist ${\rm II}_1$-factors $M_1, M_2$ with trivial $G$-action such that $x_i$ are images of $y_i\in KK^G(A_i,B_i\otimes M_i)$. By  \cite[Thm 2.14 (8)]{Kas2}, the exterior Kasparov products $y_1\otimes _\C y_2$ and $y_2\otimes _\C y_1$ coincide (under the natural flip isomorphisms) in $KK^G(A_1\otimes A_2,B_1\otimes M_1\otimes B_2\otimes M_2)$. The Kasparov products $x_1\otimes _\C x_2$ and $x_2\otimes _\C x_1$ are, by definition, the images of $y_1\otimes _\C y_2$ and $y_2\otimes _\C y_1$ respectively in the group $KK_\R^G(A_1\otimes A_2,B_1\otimes B_2)$ through a morphism from $M_1\otimes M_2$ to a  ${\rm II}_1$-factor ({\it e.g.}\/ their von Neumann tensor product) and then passing to the inductive limit. They coincide.
\end{proof}

As a direct consequence, we find:

\begin{prop}
For every pair $(A,B)$ of $G$-algebras, $KK^G_{\R}(A,B)$ is a module over the ring $KK^G_{\R}(\C,\C)$  and the Kasparov product is $KK^G_{\R}(\C,\C)$-bilinear. In particular  $KK^G_{\R}(A,A)$ is an algebra over this ring.\hfill $\square$
\end{prop}

\begin{remarks} $\,$ 
\begin{enumerate}
\item Note also that using the group morphism $G\to \{1\}$, we obtain a ring morphism $\R=KK_\R(\C,\C)\to  KK^G_{\R}(\C,\C)$ and it follows that the groups $KK^G_{\R}(A,B)$ are real vector spaces and $KK^G_{\R}(A,A)$ is an algebra over $\R$ (see also \cite[Remark 1.9]{AAS2}).

\item Let $M$ be a $\rm{II}_1$-factor with trivial $G$-action. The identity of $M$ defines an element $[\id_M]\in KK^{G}_{\R}(M,\C)$. We have a canonical map $KK^G_{\R}(A,B\otimes M ) \to KK^G_{\R}(A,B)$ by Kasparov product with this element. Let us describe it:\\
let $N$ be a $\rm{II}_1$-factor and $x \in KK^G(A,B\otimes M\otimes N)$ representing in the limit an element in $KK^G_{\R}(A,B\otimes M )$. Denote by $M\overline{\otimes} N$ the von Neumann tensor product of $M$ and $N$. Using the embedding $B\otimes M\otimes N\to B\otimes (M\overline{\otimes} N)$ we get a morphism $KK^G(A,B\otimes M\otimes N)\to KK_\R^G(A,B)$. 

Let  $\iota :\C\to M$ denote the unital inclusion. The element $\iota^*[\id_M]$ is the unit element of the ring $KK^{G}_{\R}(\C,\C)$. The element $\iota_*[\id_M]$ is therefore an idempotent in $KK^{G}_{\R}(M,M)$. It is not clear what this idempotent is.

\item The exterior product map $KK^G(A, B)\times KK^G(C,D\otimes M)\to KK^G(A\otimes C, B\otimes D\otimes M)$ induces an \lq\lq inhomogeneous product\rq\rq ~map
\begin{equation}\label{inhomogeneous}
\otimes:KK^G(A, B)\otimes KK^G_\R(C,D)\to KK^G_\R(A\otimes C, B\otimes D)\ .
\end{equation}
Clearly this is the same as computing the product in $KK_\R$ after applying the change of coefficient map $KK^G(A, B)\to KK^G_\R(A, B)$ (see also \cite[Section 1.5]{AAS2}).

\item \textbf{Real $KK$ for non-separable algebras.}
\label{nonsepKKR}
We adapt the discussion in the above Section \ref{separability} to define $KK_\R^G(A,B)$ when $G$ is second countable locally compact, and $A,B$ are generally non-separable $C^*$-algebras. 
\begin{itemize}
\item To begin with let's take $A$ separable, $B$ any $C^*$-algebra and $G$ any locally compact group. Then there exists the inductive limit
$\displaystyle\varinjlim_N KK^G(A,B\otimes N)$ with $N$ running over the space of all $\rm{II}_1$-factors (acting on a fixed separable Hilbert space) with tracial embeddings as morphisms. The existence of the limit is given by exactly the same proof in \cite{AAS2} where only the general properties of the space of $\rm{II}_1$-factors are used and there is no need of separability assumptions on $B$. We can thus put:
$$
KK_{\mathbb{R}}^G(A,B):= \varinjlim_N KK^G(A,B \otimes N).
$$
\item Assume  $A$ is separable; start with any element in $  KK_{\R}^{G}(A,B)$ represented by some $x \in KK^{G}(A,B \otimes N)$ then we can find a separable subalgebra $D \subset B\otimes N$ whose image in $KK^G(A, B\otimes N)$ is $x$. We can then find also a separable $C^*$-subalgebra $B_1 \subset B$ and an element $x_1 \in KK^{G}(A,B_1 \otimes N)$ representing $x$. This implies 
$$KK^G_{\R}(A,B)=  
\varinjlim_{\substack{B' \subset B\\  B' \textrm{separable} } }
KK^{G}_{\R}(A,B).
$$
\item By the above remark we can define $KK^G_{\R}(A,B)$ for not necessarily separable $A$ and $B$ and countable $G$. It is the analogous of $KK_{sep}^G(A,B)$ of section \ref{separability}:
 \begin{equation}\label{KKRnonsep}
 	KK^G_{sep,\R}(A,B):=  
\varprojlim_{\substack{A' \subset A\\  A' \textrm{separable} } }
KK^{G}_{\R}(A',B)=\varprojlim_{\substack{A' \subset A\\  A' \textrm{separable} } }\varinjlim_N KK^G(A',B\otimes N),
 \end{equation}
 and is also given with a well defined Kasparov product.
\end{itemize}
\end{enumerate}
\end{remarks}

\subsection{$KK$-elements associated with traces} \label{traceKKR}If $D$ is a $C^*$-algebra (with trivial $G$-action) and $\lambda$ is a tracial state on $D$, we may map $D$ in a trace preserving way into a II$_1$-factor and thus  $\lambda$ gives rise to a natural $KK$-class with real coefficients $[\lambda]\in KK_{\R}(D,\C)$.

\medskip
Let us make some remarks.

\begin{remarks}\label{Remtraces} Let $D$ be a $G$-algebra and $\lambda$ a tracial state on $D$. 
\label{remarksKKG}
\begin{enumerate}
\item  If the action of $G$ on $D$ is inner - \textit{i.e}.\ through a continuous morphism of $G$ into the set of unitary elements of $D$, then $D$ with the $G$ action is Morita equivalent to the algebra $D$ with trivial $G$ action, and thus  $[\lambda]\in KK^G_{\R}(D,\C)$ is still well defined.

\item Recall that when $G$ is discrete and acts trivially on $B$ the groups $KK^G(A, B)$ and $KK(A\rtimes G, B)$ coincide (the crossed product $A\rtimes G$ is the maximal one).  The same isomorphism holds at $KK_\R^G$ level. 
\item If $G$ is discrete and $\lambda $ is invariant, then the crossed product algebra $D\rtimes G$ carries the natural (dual) trace associated with $\lambda$. We have an equivariant inclusion $D\to D\rtimes G$ where the action of $G$ is now inner on $D\rtimes G$. We thus still get an element $[\lambda]\in KK^G_{\R}(D,\C)$.

\item Let $\Gamma$ be a discrete group. Then by the discussion above every tracial state on $C^*\Gamma$ defines an element of $KK_\R(C^*\G,\C)=KK_\R^{\G}(\C,\C).$ Let $\sigma,\sigma'$ be tracial states on $C^*\G$.  Let $\delta:C^*\G \longrightarrow C^*\G \otimes C^*\G$ denote the coproduct; then we can write $[\sigma]\otimes [\sigma'] = [\sigma.\sigma']$ where $\sigma.\sigma'$ is the tracial state $\sigma.\sigma':=(\sigma \otimes \sigma') \circ \delta$. 
We have $\sigma.\sigma'(g)=\sigma(g)\sigma'(g)$ on every $g\in \G$.

In particular, the canonical group trace  $\tr_\Gamma$ which satisfies  $\tr_\Gamma(e)=1$ and  $\tr_\Gamma(g)=g$ for $g\ne e$ absorbs every tracial state:  $[\sigma]\otimes[\tr_\Gamma]= [\tr_\Gamma]$.
\end{enumerate}
\end{remarks}

\subsection{Descent morphism}
Let us describe the descent morphism in $KK_\R$.

Let $\G$ be a discrete group, then there is a descent map 
\begin{equation}
\label{eq:desc_max}
J^{\Gamma}: KK^\G_{\R}(A,B) \longrightarrow KK_{\R}(A \rtimes \Gamma,B \rtimes \G),
\end{equation} 
with respect to the maximal crossed product, which is induced by the classical Kasparov descent map on $KK^\Gamma$. It is natural with respect to the Kasparov product and satisfies  
$J^{\G}(1_A)=1_{A \rtimes \G}$.

Indeed let $M \to N$ be a morphism of $C^*$-algebras, both with trivial $\G$ action, and let $B$ a $\G$-algebra; we have a commutative diagram
\begin{equation}\label{crosstensor}
\xymatrix{(B\otimes M) \rtimes \Gamma \ar[d]\ar[r] & (B\otimes N) \rtimes \Gamma \ar[d] \\
( B \rtimes \Gamma)\otimes M \ar[r] & (B\rtimes \G)\otimes N, 
}
\end{equation}
which follows from the universal properties of the full crossed product and the fact that $\G$ acts trivially on $M$ and $N$. When $M \hookrightarrow N$ is a unital embedding of $\rm{II}_1$-factors, applying the $KK$ functor and the Kasparov descent map before passing to the limit defines \eqref{eq:desc_max}. 

\medskip We have a commutative diagram similar to \eqref{crosstensor} involving reduced crossed products. Actually, when taking reduced crossed products (and minimal tensor products) the vertical arrows are isomorphisms: 
 let $B$ be acting  faithfully on the Hilbert space $H_1$ and $M$ on the Hilbert space $H_2$; then both $(B\otimes M) \rtimes_{r} \Gamma$ and $(B\rtimes_{r} \Gamma)\otimes M$ are canonically and faithfully represented in the Hilbert space $H_1\otimes   H_2 \otimes \ell^2\Gamma$ - and the same with $N$ instead of $M$. 

Therefore, in the same way we define a reduced descent morphism 
\begin{equation}
\label{eq:desc_red}
J^\Gamma_r :KK_{\R}^\G(A,B) \longrightarrow KK_{\R}(A \rtimes_r \Gamma, B\rtimes_r \Gamma)\ .
\end{equation}

\section{The $\tau$-element and the $\tau$-part}
Let $\Gamma$ be a discrete group; the $\tau$-\emph{element} of $\Gamma$ is associated to $\Gamma$ by means of its canonical trace $\operatorname{tr}_\Gamma:C^*\Gamma \rightarrow \C$.
Indeed applying point (4) of Remark \ref{Remtraces} to $\operatorname{tr}_\Gamma$ we get a natural class that we call $[\tau]\in KK^\Gamma_{\R}(\C,\C)$. 

Two algebraic properties of $[\tau]$ are immediate to check: $[\tau]$ is an idempotent, \textit{i.e}.\ $[\tau]\otimes [\tau]=[\tau]$  (see \cite{AAS2}); furthermore, $[\tau]$ is \emph{central} in $KK_{\R}^\Gamma$, \textit{i.e}.\  for any  $\Gamma$-algebras $A,B$ and any $[x] \in KK_{\R}^\Gamma(A;B)$ then  $[x]\otimes [\tau]= [\tau]\otimes [x]$, as follows by Lemma \ref{commute}.

Since $\tau$ is fundamental for what follows, it is worth giving a more detailed description:
\begin{itemize}
\item one starts with any trace preserving morphism $\varphi_N:C^*\Gamma \longrightarrow N$ to a $\rm{II}_1$-factor which defines a class in $KK(C^*\Gamma,N)$ and in the limit a class $[\varphi_N]\in KK_{\R}(C^*\Gamma;\C).$
\item Consider $\Gamma$ acting trivially on $\C$ and apply the canonical isomorphism $KK_{\R}(C^*\G,\C)\simeq KK_{\R}^\G(\C,\C)$ to the class $[\varphi_N]$. The element we get is $[\tau]$.
\item Concretely $[\tau]$ is represented in the limit by the class in $KK^\Gamma(\C,N)$ of the $\G$-$(\C,N)$ bimodule which is $N$ (considered as a Hilbert $N$-module) with the $\Gamma$-action $\gamma \cdot n= \varphi_N(\gamma) n$. If $\Gamma$ is an i.c.c. group, one can take  $N=L\Gamma$ to be the group von Neumann algebra of $\Gamma$. More generally one can for instance embed $\Gamma$ in an i.c.c. group $\Gamma'$ and take  $N$ to be the group von Neumann algebra of $\Gamma'$.
\end{itemize}

\subsection{The $\tau$-part of $K$-theory}
\subsubsection{The $\tau$-part of $KK_\R^\Gamma$ and of crossed products}
The action of the idempotent $[\tau]$ suggests the following definition.
\begin{definition}
 We will denote by $KK_{\R}^\Gamma(A,B)_{\tau}$ the image of the idempotent $[\tau]$ acting on $KK_{\R}^\Gamma(A,B)$ and call it the \emph{$\tau$-part of $KK_\R^\G(A,B)$}.
 \end{definition}

It is natural to define  the $\tau$-part also for the $K$-theory of group $C^*$-algebras (or more generally crossed products). To do so, we consider the action of the element $J^\Gamma([\tau])$ obtained by applying the descent morphisms, as follows. 

\begin{definition}[$\tau$-part of crossed products]
Let $J^\Gamma$  and $J^\Gamma_r$ be the descent maps defined in \eqref{eq:desc_max}, \eqref{eq:desc_red}. Let $A$ be a $\Gamma$-$C^*$-algebra. Let $1_A$ denote the unit of $KK^\Gamma(A,A)$. We set
\begin{align}
&[\tau^A]_{\text{max}}:=J^\Gamma(1_A\otimes [\tau])\in KK_\R(A\rtimes \Gamma, A\rtimes\Gamma)\\
&\;\;\;[\tau^A]_{r}:=J^\Gamma_r(1_A\otimes [\tau])\in KK_\R(A\rtimes_r\Gamma, A\rtimes_r\Gamma)\ .
\end{align} 
When $A=\C$ we simply write $[\tau]_{\text{max}}$ and $[\tau]_{r}$.

Let now $D$ be any $C^*$-algebra. We call \emph{the $\tau$-part of} $KK_\R(D,A\rtimes_r\Gamma)$ the image of the idempotent $[\tau^A]_{r}$ acting by right multiplication on $KK_{\R}(D,A\rtimes_r\Gamma)$:
$$KK_\R(D,A\rtimes_r\Gamma)_{\tau}:=\textrm{Image}    \big{\{} \cdot \otimes    [\tau^A]_{r}:   KK_{\R}(D,A\rtimes_r\Gamma) \longrightarrow KK_{\R}(D,A\rtimes_r\Gamma)                  \big{\}}.$$  
Analogously,  \emph{the $\tau$-part of} $KK_\R(D,A\rtimes\Gamma)$, denoted $KK_\R(D,A\rtimes\Gamma)_{\tau}$, is the image of  $\,[\tau^A]_{\text{max}}$ acting by left multiplication on $KK_{\R}(D,A\rtimes \Gamma)$.
When $D=\C$ we abbreviate with $K_{*,\R}(C^*\Gamma)_{\tau}$ and  $K_{*,\R}(C_r^*\Gamma)_{\tau}$. 
\end{definition}

\subsubsection{Reduced versus maximal crossed products}
\label{section:redmax}
The name $\tau$-part is unambiguous; indeed the next Proposition shows that $KK_\R(D,A\rtimes \Gamma)_{\tau}$ and $KK_\R(D,A\rtimes_r\Gamma)_{\tau}$ are canonically isomorphic.

Let $\lambda^A:A\rtimes\Gamma\to A\rtimes_r \Gamma$ be the natural morphism and denote by $[\lambda^A]\in KK(A\rtimes\Gamma, A\rtimes_r \Gamma)$ its class. 

There is also a natural morphism
$\Delta^A:A\rtimes_r\Gamma\to (A\rtimes \Gamma)\otimes C^*_r\Gamma$ induced by the coproduct. Indeed, the coproduct gives a morphism $\Delta_{max}^A:A\rtimes \Gamma\to (A\rtimes \Gamma)\otimes C^*_r(\Gamma)$ defined by the covariant morphism $a\mapsto a\otimes 1$ and $g\mapsto g\otimes \lambda _g$.
Let $A\rtimes \Gamma$ be faithfully represented on a Hilbert space $H$. Then we obtain a faithful representation $\pi$ of $(A\rtimes \Gamma)\otimes C^*_r(\Gamma)$ on $H\otimes \ell^2(\Gamma)$. The composition $\pi\circ \Delta_{max}^A$ is the canonical faithful representation of the reduced crossed product. It follows that the morphism $\Delta _{max}^A$ factors through $A\rtimes_r \Gamma$: there is a (faithful) morphism $\Delta^A:A\rtimes_r\Gamma\to (A\rtimes \Gamma)\otimes C^*_r\Gamma$ such that $\Delta_{max}^A=\Delta^A\circ\lambda^A$.

Composing with a trace preserving embedding of $C^*_r\Gamma$ in a $\rm{II}_1$-factor $N$, we thus obtain a class $[\Delta_\tau^A]=[\Delta^A]\otimes _{C^*_r\Gamma}[\tau]\in KK_\R(A\rtimes_r\Gamma, A\rtimes \Gamma)$.

\begin{prop}{\emph{(}Cf. \cite[proof of Remark 2.4]{AAS2}\emph{)}}  
\label{prop:redmax}
We have
	\begin{align}
\label{eq:factAcrossG}
&[\lambda^A]\otimes [\Delta^A_\tau] =[\tau^A]_{\operatorname{max}}\in KK_\R(A\rtimes\Gamma, A\rtimes\Gamma), \quad  \\ \nonumber
&[\Delta^A_\tau]\otimes [\lambda^A] = [\tau^A]_{r}  \;\;\;\;\in KK_\R(A\rtimes_r\Gamma, A\rtimes_r\Gamma)\ .
	\end{align} 
Then for every $C^*$-algebra $D$ the product with $[\Delta^A_\tau]$ induces a canonical isomorphism
$$
\xymatrix{
KK_\R(D, A\rtimes_r\Gamma)_{\tau} \ar[r]^{{\scalebox{1.5}[1]{$\simeq$}}}&  KK_\R(D, A\rtimes \Gamma)_{\tau}}.
$$ 

\end{prop}
\begin{proof}
Represent $[\tau]$ using a trace preserving morphism $\varphi: C_r^*\Gamma \longrightarrow N$. Then $[\tau^A]_{\operatorname{max}}$ is given by the composition $(\id\otimes \varphi)\circ \Delta_{max}^A=(\id\otimes \varphi)\circ \Delta^A\circ\lambda^A$. The first equality follows.

In the same way, $[\tau^A]_r$ is given by the composition $(\id\otimes \varphi)\circ \Delta_{r}^A$, where $\Delta_{r}^A$ is the usual coproduct $A\rtimes_r\Gamma\to (A\rtimes_r\Gamma)\otimes C^*_r\Gamma$. The second equality follows since $\Delta_r^A=(\lambda^A\otimes \id)\circ \Delta^A$.
\end{proof}

We can interpret the above fact by saying that $[\tau]$ belongs to the ``$K$-amenable part of $KK^\Gamma_\R(\C,\C)$'' in the spirit of \cite{Cuntz}.

\subsubsection{Functoriality}\label{functorial}
We remark here a  functoriality property of the $\tau$-parts that we need later. 

Let $A$ and $B$ be $\Gamma$-algebras and $x\in KK_{\R}^\Gamma (A,B)$. 

\begin{prop}
The map $\,\cdot \:\otimes_{J_r(x)}:KK_{\R}(\C, A\rtimes_r \G) \longrightarrow KK_{\R}(\C, B\rtimes_r \G)$ preserves the $\tau$-parts inducing a map
$KK_{\R}(D, A\rtimes_r \G)_\tau \longrightarrow KK_{\R}(D, B\rtimes_r \G)_\tau.$
In particular if $f:A\to B$  is  an equivariant morphism,  the induced map $f_*:KK_{\R}(\C, A\rtimes_r \G) \longrightarrow KK_{\R}(\C, B\rtimes_r \G)$ preserves the $\tau$-parts.
\end{prop}
\begin{proof}
Let $y\in KK_{\R}(D; A\rtimes_r \G)_\tau$. Then $y=y\otimes J_r(1_A\otimes [\tau])$ and $y\otimes J_r(x)=y\otimes J_r([\tau]\otimes x)=y\otimes J_r(x\otimes [\tau])\in KK_{\R}(D; B\rtimes_r \G)_\tau.$
\end{proof}

This discussion holds of course also for the full crossed product.

\subsection{Naturality of the $\tau$-part of $KK_\R(\C, C^*_r\Gamma)$ with respect to group morphisms}
\label{naturalitytau}

Let $\varphi:\Gamma_1 \to \Gamma_2$ be a morphism of groups. It defines a morphism $\varphi:C^*\Gamma_1\to C^*\Gamma_2$ and therefore an element $[\varphi]\in KK(C^*\Gamma_1,C^*\Gamma_2)$. This morphism is not defined in general at the level of reduced $C^*$-algebras (if $\ker \varphi$ is not amenable) unlike the left hand side $K_*^{\rm top}(\Gamma)$ of the Baum--Connes assembly map. 

Using the $\tau$ elements, we can easily bypass this difficulty.

We denote by $[\tau_1]\in KK^{\Gamma_1}_\R(\C,\C)$ and $[\tau_2]\in KK^{\Gamma_2}_\R(\C,\C)$ the corresponding $\tau$-elements. 

Define $[\varphi]_r\in KK_\R(C_r^*\Gamma_1,C_r^*\Gamma_2)$ by putting $[\varphi]_r=(\lambda_2\circ \varphi)_*[\Delta_{\tau_1}]$ where $[\Delta_{\tau_1}]=[\Delta^\C_{\tau_1}]\in KK_\R(C_r^*\Gamma_1,C^*\Gamma_1)$ was defined in section \ref{section:redmax}.

By Kasparov product by $[\varphi]_r$, we obtain a linear map 
$
KK_{\R}(\C,C_r^*\Gamma_1)\overset{\otimes [\varphi]_r}{\longrightarrow} KK_{\R}(\C,C_r^*\Gamma_2).
$

\begin{claim} 
$[\phi]_r=J_r^{\Gamma_1}([\tau_1])\otimes [\varphi]_r\otimes J_r^{\Gamma_2}([\tau_2])$. In particular, the map $\otimes [\varphi]_r$ preserves the $\tau$-parts.  
\end{claim}

 \medskip Indeed, \begin{itemize}
\item $[\Delta_{\tau_1}]=J_r^{\Gamma_1}([\tau_1])\otimes [\Delta_{\tau_1}]$, and thus $[\phi]_r=J_r^{\Gamma_1}([\tau_1])\otimes [\varphi]_r$.
\item The element $[\varphi]_r$ is the class of a morphism $C^*_r(\Gamma_1)\to C^*_r(\Gamma_2)\otimes N_1$ given by $\delta_{g_1}\mapsto \delta_{\varphi({g_1})}\otimes j_1(\delta_{g_1})$ (for ${g_1}\in \G_1$) where $N_1$ is any $\rm{II}_1$-factor and $j_1:C^*_r\Gamma_1\to N_1$ is any trace preserving morphism. 

Let $j_1:C^*_r\Gamma_1\to N_1$ and $j_2:C^*_r\Gamma_2\to N_2$ be trace preserving inclusions into $\rm{II}_1$-factors. The element $[\varphi]_r\otimes J_r^{\Gamma_2}([\tau_2])$ corresponds to the morphism $C^*_r(\Gamma_1)\to C^*_r(\Gamma_2)\otimes N_2\overline \otimes N_1$, given by $\delta_{g_1}\mapsto \delta_{\varphi({g_1})}\otimes j_2(\delta_{\varphi({g_1})})\otimes j_1(\delta_{g_1})$. But the map $\delta_{g_1}\mapsto j_2(\delta_{\varphi({g_1})})\otimes j_1(\delta_{g_1}) $ is a trace preserving inclusion of $C^*_r\Gamma_1$ into the $\rm{II}_1$-factor  $N_2\overline\otimes N_1$, and thus $[\varphi]_r\otimes J_r^{\Gamma_2}([\tau_2])=[\varphi]_r$.
\end{itemize}

Hence we have shown

\begin{prop}
\label{naturality}
A morphism of discrete groups $\varphi:\Gamma_1 \to \Gamma_2$ naturally induces a linear map 
$$
K_{*,\R}(C_r^*\Gamma_1)_\tau\longrightarrow K_{*,\R}(C_r^*\Gamma_2)_\tau\ .
$$

\end{prop}

\subsection{The $\tau$-Baum--Connes map}  
\subsubsection{Classifying spaces}\label{Classspace}
Let $\Gamma$ be a discrete group and $A$ a $\Gamma$-algebra.
The group $K_{*}^{\textrm{top}}(\Gamma; A)$ is defined \cite{BCH} by $$
K_{*}^{\textrm{top}}(\Gamma; A)=\varinjlim_{Y }KK_*^\Gamma(C_0(Y), A).
$$
In this direct limit of Kasparov equivariant homology groups, all the proper cocompact   $\Gamma$-invariant subspaces $Y\subset \underline E\Gamma$ of the classifying space for proper actions $\underline{E}\Gamma$ are taken into account. 

The universal property of $\underline E\Gamma$ ensures that every proper $\G$-space $Z$ has a $\Gamma$-map to $\underline E\Gamma$ which is unique up to homotopy. {Therefore $K_{*}^\Gamma(\underline E\Gamma; A)$ is the limit of the inductive system $KK_{*}^\Gamma(C_0(Z), A)$ where $Z$ runs over proper  and cocompact $\Gamma$-spaces.} 
We can therefore use the following notation: for a pair $(Y, y)$ where $Y$ is a proper and cocompact $\Gamma$-space and $y\in KK^\Gamma(C_0(Y), A)$, we denote by $\llbracket \, Y, y\, \rrbracket$ its associated class in the inductive limit.

Concerning real coefficients, we define:
$$
K_{*,\mathbb{R}}^{\textrm{top}}(\Gamma; A)=\varinjlim_{\substack{Y\subset \underline E\Gamma\\ Y \;\Gamma\text{-compact}}} KK_{*,\R}^\Gamma(C_0(Y), A).
$$
Notice this is an iterated limit, first over factors, then on subsets $Y$. It exists because $Y \longmapsto KK^{\Gamma}_{\R}(C_0(Y),A)$ is a directed system with values groups. Also by Remark 1.6 (4) in \cite{AAS2}, the opposite iterated limit $$\varinjlim_N \varinjlim_{\substack{Y\subset \underline E\Gamma\\ Y \;\Gamma\text{-compact}}} KK_{*}^\Gamma(C_0(Y), A\otimes N)=\varinjlim_N  K_{*}^{\operatorname{top}}(\Gamma,A\otimes N).
$$
exists. Here $N$ ranges over a fixed space of $\rm{II}_1$-factors acting on a separable Hilbert space and with morphisms $N \longrightarrow M$ which are trace preserving embeddings. We can show that these two limits coincide and give a well defined group 
$K^{\operatorname{top}}_{*,\R}(\Gamma;A)$. Indeed we are just taking the limits in the two entries of the covariant bifunctor $(Y,N) \mapsto KK^\Gamma(C_0(Y), A\otimes N)$ defined on the corresponding product category. Generalising the notation above, we are thus allowed to represent elements in $K_{*,\R}^{\operatorname{top}}(\Gamma;A)$ as classes $\llbracket \, Y, y\, \rrbracket $ with $Y$ proper $\Gamma$-compact $\Gamma$-space  and $y \in KK^\Gamma(C_0(Y),A \otimes N)$.

\subsubsection{The Baum--Connes map with real coefficients}  
Let $A$ a be $\Gamma$-algebra. Recall that the Baum--Connes assembly map $\mu^A: K^{\textrm{top}}_*(\Gamma;A)\longrightarrow K_{*}(A\rtimes_r \Gamma)$ assigns to $\llbracket \,Y, \xi\,\rrbracket \in K^{\textrm{top}}_*(\Gamma;A) $, as above the element
$$
\mu^A(\llbracket \,Y, \xi\,\rrbracket)=p_Y \otimes j_r^\Gamma (\xi)
$$ 
where $p_Y\in KK(\C, C_0(Y)\rtimes_r \Gamma)$ is the Kasparov projector. 

The map $\mu^A$ factors through the $K$-theory of the maximal crossed product $K_{*}(A\rtimes \Gamma)$. 
We will sometimes write just $\mu$ instead of $\mu ^A$ if $A=\C$ or when there is no ambiguity on the $\Gamma$-algebra $A$.

\medskip
Replacing $A$ with $A\otimes N$ for a $\rm{II}_1$-factor $N$ (with trivial $\G$-action), and using the isomorphism 
$(A\otimes N) \rtimes_r \Gamma \simeq (A\rtimes_r \Gamma )\otimes N$ defines a collection of assembly maps
$K^{\operatorname{top}}_{*}(\Gamma; A\otimes N) \longrightarrow K_*((A\rtimes_r \Gamma)\otimes N)$
passing to the limit to an assembly map in $K_{\R}$:
\begin{equation}\label{realassemby}
	\mu_{\R}:K^\textrm{top}_{*,\R}(\Gamma;A) \longrightarrow K_{\R}(A \rtimes_r \Gamma).
\end{equation}
By construction $\mu_{\R}$ is formally defined by the same recipe as $\mu.$
\begin{lemma}\label{muR}
The assembly map $\mu_{\R}$ is given by first applying the real reduced descent map and then taking the $KK_{\R}$-product with $p_Y\in KK(\C, C_0(Y)\rtimes_r \Gamma)$ followed by the direct limit on $Y$:
$$
\xymatrix{KK^{\Gamma}_{\R}(C_0(Y);A)\ar[r]^-{J^\Gamma_r} &KK_{\R}(C_0(Y)\rtimes_r \G;A \rtimes_r \Gamma)\ar[r]^-{p_Y \otimes \cdot} &KK_{\R}(\C;A \rtimes \Gamma)}\ .
$$
\end{lemma}
\begin{proof}
We know that we can interchange the limits in our definition of $K^\textrm{top}_{\R}(\Gamma;A)$. Then the Lemma follows immediately 
by the definition of the descent morphism in $KK_{\R}$ (\ref{eq:desc_max}) and of the $KK_{\R}$-Kasparov product. Indeed starting from \eqref{realassemby}, and with self-explanatory notation:
$$\mu_{\R} = \varinjlim_N \mu^{A \otimes N} = \varinjlim_N \varinjlim_Y p_Y \otimes J^\Gamma_r =\varinjlim_Y  p_Y \otimes  \varinjlim_N  J^\Gamma_r\ .
$$
\end{proof}

\subsubsection{The $\tau$-Baum--Connes map}  
We now define a map 
$\mu_\tau: K^{\textrm{top}}_{\R,*}(\Gamma;A)_{\tau}\longrightarrow  K_{\R,*}(A\rtimes_r \Gamma)_{\tau}$ between the corresponding $\tau $-parts of the real $K$-theory.

Let us remark that if $A,B$ are $\Gamma$-algebras,  $z\in K^{\textrm{top}}_{\R,*}(\Gamma;A)$ and $y\in KK^\Gamma_\R(A,B)$ then we have  $\mu^B_\R(z\otimes y)=\mu^{A}_{\R}(z)\otimes J_r^\Gamma(y)$.
Taking $B=A$, we find that the real assembly map $\mu^A_{\R}$ is $KK^\Gamma_\R(A,A)$ linear from the right $KK^\Gamma_\R(A,A)$-module $K^{\textrm{top}}_{\R,*}(\Gamma;A)$ to the  right $KK^\Gamma_\R(A,A)$-module $KK_\R(\C,A\rtimes_r \Gamma)$ (where $KK^\Gamma_\R(A,A)$ acts via the ring morphism $J^\Gamma_r:KK^\Gamma_\R(A,A)\to KK_\R(A\rtimes_r \Gamma,A\rtimes_r \Gamma)$.
In particular, using the morphism $y\mapsto 1_A\otimes y$ from $KK^\Gamma_\R(\C,\C)\to KK^\Gamma_\R(A,A)$, it follows that $\mu^A_\R$ is $KK^\Gamma_\R(\C,\C)$-linear.

We therefore find a map $\mu_{\tau}$ filling the commutative diagram:
$$\xymatrix{K^{\textrm{top}}_{*,\R}(\Gamma;A)\ar[d]_{[\tau]} \ar[r]^-{\mu_{\R}}& K_{\R}(A \rtimes_r \Gamma)\ar[d]^{[\tau]_r}\\
K^{\textrm{top}}_{*,\R}(\Gamma;A)_{\tau} \ar[r]^{\mu_{\tau}}&K_{\R}(A\rtimes_{r}\Gamma)_{\tau}.  }$$
Here on the right vertical arrow there is the product with $[\tau]_r:=J_r^\Gamma(1_A\otimes [\tau])$.
It is straightforward to check directly the existence of $\mu_{\tau}$; indeed using Lemma \ref{muR} we compute for every $z \in K_{*,\R}^{\operatorname{top}}(\Gamma; A)$ represented as a class $\llbracket Y,y \rrbracket$ with $y \in KK_{\R}^{\Gamma}(C_0(Y),A)$:
$$\mu_{\R}(z) \otimes [\tau]_r = p_Y \otimes J^{\Gamma}_r (y  \otimes [\tau]) = p_Y \otimes J^{\Gamma}_r (  [\tau] \otimes y ) =  \mu_{\R}([\tau]\otimes z)\ .$$ 
It follows that $\mu_{\R}$ descends to a map on the $\tau$-parts simply given by:
$$
\mu_\tau(x \otimes [\tau]):=\mu_{\R}(x) \otimes [\tau]_r, \quad x \in K^{\textrm{top}}_{*,\R}(\Gamma;A).
$$

\begin{definition} 
\label{def:tauBC}% 
We call $\mu_\tau: K^{\textrm{top}}_{*,\R}(\Gamma;A)_{\tau}\longrightarrow  K_{*,\R}(A\rtimes_r \Gamma)_{\tau}$ the \emph{$\tau$-Baum--Connes map}.
\end{definition}

We can state the following $\tau$-form of the Baum--Connes conjecture with coefficients:
\begin{conjecture}
The map $\mu_\tau: K^{\textrm{top}}_{*,\R}(\Gamma;A)_{\tau}\longrightarrow   K_{*,\R}(A\rtimes_r \Gamma)_\tau$ is bijective.
\end{conjecture}

\medskip

\begin{theorem} \label{thm:BCImpliesBCtau}
 Fix a $\G-C^*$-algebra $A$; if the Baum--Connes map $$\mu:K^\textrm{top}_*(\Gamma;A\otimes N) \longrightarrow K_*(A\rtimes_r \Gamma \otimes N)$$ with coefficients in $A\otimes N$ is injective (surjective) for any choice of a $\rm{II}_1$-factor $N$, then\begin{enumerate}
\item   the corresponding $\mu_\R: K^{\textrm{top}}_{*,\R}(\Gamma;A) \longrightarrow K_{\R}(A\rtimes_r \Gamma)$ is injective (surjective);
\item   the corresponding $\mu_\tau: K^{\textrm{top}}_{*,\R}(\Gamma;A)_{\tau}\longrightarrow K_{\R}(A\rtimes_r \Gamma)_\tau$ is injective (surjective). 
\end{enumerate}
\end{theorem}
\begin{proof} We prove (1). Statement (2) follows since the morphism $\mu_\R$ is $KK_\R^\Gamma(\C,\C)$ linear.

To establish the surjectivity statement just note that every element of $K_{*,\R}(A\rtimes_r \Gamma)$ is the image of an element of $K_*((A\rtimes_r \Gamma) \otimes N)=K_*((A\otimes N)\rtimes_r \Gamma) $ and thus, by surjectivity of $\mu^{A\otimes N}$, from an element in $K_*^{\rm top}(\Gamma;A\otimes N)$.

Let us now establish the injectivity statement. Let $x \in K^{\textrm{top}}_{*,\R}(\Gamma;A)$ be such that $\mu_{\R}^A(x) =0 \in K_{*,\R}(A \rtimes_r \Gamma).$

The element $x$ comes from an element $x_0\in K^{\textrm{top}}_{*}(\Gamma;A\otimes N_0)$. The fact that $x\in \ker \mu_\R$ means that $\mu(x_0)$ is zero in some $K_*((A \rtimes_r \Gamma)\otimes N_2)$, where $N_1\subset N_2$. By injectivity of $\mu^{A\otimes N_2}$ it follows that the image of $x_0\in K_*^{\rm top}(\Gamma;A\otimes N_2)$ is $0$ and therefore $x=0.$
\end{proof}

\section{Action of $[\tau]$ on the classifying space}

\subsection{The group  $K_{*}^\Gamma(E\Gamma; A)$.}

Replacing $\underline E\Gamma$ by the usual classifying space $E\Gamma$ in the definition of $K^{\operatorname{top}}_{*}(\Gamma;A)$ we get the construction of the group  $K_{*}^\Gamma(E\Gamma; A)$: 

\begin{definition}The equivariant $K$-homology with $\Gamma$-compact supports of the standard classifying space for free and proper actions is the group
$$
K_{*}^\Gamma(E\Gamma; A)=\varinjlim_{\substack{Y\subset E\Gamma\\ Y \;\Gamma\text{-compact}}} KK_{*}^\Gamma(C_0(Y), A).
$$
\end{definition}

{In the same way as for $\underline E\Gamma$, the universal property of $E\Gamma$ ensures that every free and proper $\G$-space $Z$ has a $\Gamma$-map to $E\Gamma$ which is unique up to homotopy. Therefore $K_{*}^\Gamma(E\Gamma; A)$ is the limit of the inductive system $KK_{*}^\Gamma(C_0(Y), A)$ where $Y$ runs over free, proper  and cocompact $\Gamma$-spaces.}

{
We denote by $\llangle Y,y\rrangle$ the class in $K_{*}^\Gamma(E\Gamma; A)$ of a pair $(Y, y)$ where $Y$ is a free proper and cocompact $\Gamma$-space and $y\in KK^\Gamma(C_0(Y), A)$.}

\subsection{Atiyah's theorem and $B\Gamma$}
In \cite{AAS} we proved that Atiyah's theorem for a covering space $\widetilde{V} \rightarrow V$ with deck group $\Gamma$ is equivalent to the triviality in $K$-theory of the Mishchenko bundle $\widetilde{V} \times_{\Gamma} N$ constructed  from a morphism $\Gamma \rightarrow U(N)$ into a $\rm{II}_1$-factor $N$ (see also the recent \cite{KP}).
This inspired the definition of \lq\lq $K$-theoretically free and proper\rq\rq (KFP) algebras as the $\G$-algebras where $[\tau]$ acts as the identity \cite{AAS2}.
\begin{definition}For a $\G$-algebra $A$ we say that it has the (KFP) property if the equation $1_A^\G \otimes [\tau]=1_{A,\R}^\G$ holds in $KK^\G_{\R}(A,A)$. 
\end{definition} 
 Here $1_A^\G$ and $1_{A,\R}^{\G}$ are the units of the rings $KK^\G(A,A)$ and $KK^{\G}_{\R}(A,A)$ respectively.
 \begin{remark}
 \label{kfp2} Note that $A$ has the (KFP) property if and only if $KK_{*,\R}^\Gamma(A, A)_{\tau}=KK_{*,\R}^\Gamma(A, A)$. 
 Also, if $A$ has the (KFP) property then by Lemma \ref{commute} we know that $[\tau]$ acts as the unit element also on $KK^{\G}_{\R}(A,B)$ and $KK^{\G}_{\R}(B,A)$ for every $C^*$-algebra $B$. 
 \end{remark}
 The main examples of (KFP) algebras are the free and proper ones in the language of Kasparov \cite{Kas2} (see \cite[Theorem 3.10]{AAS2}).  Of course $C_0(Y)$ for a cocompact free and proper $\G$-space $Y$ is (KFP). Put in another way:
\begin{prop}
\label{prop:atiyah}
Let $A$ be any $\G$ - $C^*$-algebra. The action of $[\tau]$ on $K_{*,\mathbb{R}}^\G(E\G; A)$ is the identity. 
In particular $$K_{*,\mathbb{R}}^\G(E\G;A)_{\tau}=K_{*,\mathbb{R}}^\G(E\G;A)$$
\end{prop}
\begin{proof}
The $K$-homology group $K_{*,\mathbb{R}}^\G(E\G;A)$ is defined as a limit over all the cocompact $\G$-invariant pieces. It is thus sufficient to prove that for every cocompact free and proper $\G$-space $Y$ the element $[\tau]$ acts as the identity on $KK^{\G}_{\R}(C_0(Y),A)$. This is true because $C_0(Y)$ is (KFP) and by Remark \ref{kfp2}.
\end{proof}

%%%%%% %%%%%%%%%%%% %%%%%%%%%%%% %%%%%%%%%%%% %%%%%%
%% %%%%%        comparison_tau.tex

\section{Comparison between  $K^\Gamma_{*,\R}(E\Gamma;A) $ and $K^{\operatorname{top}}_{*,\R}(\Gamma;A)$}

In this section we use $[\tau]$ to compare the $K$-homology of $E\Gamma$ with the one of $\underline{E}\Gamma$ at the level of real coefficients. In particular for any $\Gamma$-algebra $A$ we show that the natural  map $\sigma:E\Gamma \longrightarrow \underline{E}\Gamma$ induces an isomorphism 
$$
\xymatrix{ K^\Gamma_\R(E\Gamma; A) \ar[r]^-{\sigma_*}_{\simeq\;\;}&  K^{\textrm{top}}_{*,\R}(\Gamma;A)_{\tau}}.
$$
In the case $A=\C$  this will be used in Section \ref{sec:Novikov} and implies an isomorphism
$$
\xymatrix{K_*(B\Gamma)\otimes \R \ar[r]^-{}_{\;\;\simeq}&  K^{\operatorname{top}}_{*,\R}(\Gamma)_{\tau}}.
$$
{Note that $\sigma_*\big(\llangle Y,y\rrangle\big)=\llbracket Y,y\rrbracket$ for every free, proper cocompact $\Gamma$-space $Y$ and every $y\in KK_{*,\R}^\Gamma(C_0(Y), A)$.}

We construct the inverse of $\sigma_*$ using any compact probability space where $\Gamma$ acts probability measure preserving (p.m.p.) and sufficiently freely.

 \begin{prop}
\label{prop:measure}
Let $(X,m)$ be any compact p.m.p.\ $\Gamma$-space. The invariant measure $m$ defines a dual trace $t_{m} : C(X)\rtimes \Gamma \longrightarrow \C$ hence an element  $[t_{m}] \in KK^\Gamma_{\R}(C(X),\C)$. We have  $i_X^*[t_m]=[\tau]$ where $i_X: \C\hookrightarrow C(X)$ is the unital inclusion and $[t_m]\otimes [\tau]= [t_m]$.
\end{prop}
\begin{proof}
The first point is an immediate consequence of Remark \ref{remarksKKG}. By definition of a dual trace, we have $t_m\circ i_X=\tau$. Finally, $[t_m]\otimes [\tau]$ is the class of the trace $(t_m\otimes \tau)\circ \,\delta_X$ on $C(X)\rtimes \Gamma $, where $\delta_X:C(X)\rtimes \Gamma \to (C(X)\rtimes \Gamma )\otimes C^*\Gamma$ is the coaction. Obviously $(t_m\otimes \tau)\circ \delta_X=t_m$ since $t_m$ is a dual trace.
\end{proof}

\medskip

\subsection{Property (TAF)}
 \begin{definition}
A compact space $X$ with an action of $\Gamma$ such that every torsion element acts freely is said to have property (TAF),  \emph{torsion acts freely}.
\end{definition}

Note that if $X$ is (TAF) and $Y$ is a proper $\Gamma$-space, then $Y\times X$ is free and proper with respect to the diagonal action. We use this simple remark to map naturally $K^{\textrm{top}}_{*,\R}(\Gamma;A) $ to the \lq\lq free and proper part\rq\rq $\;K_{*, \R}^\Gamma(E\Gamma; A)$, for every $\Gamma$-algebra $A$.
\begin{prop}\label{prop5.3}
 Let $(X,m)$ be any compact p.m.p.\ $\Gamma$-space with (TAF) property, and $[t_m]\in KK^\Gamma_{\R}(C(X),\C)$ the class constructed above.  
\begin{enumerate}
\item The  product by $[t_m]$ induces a well-defined map 
\begin{equation}
\label{eq:t_Xmap}
t_{(X,m)}:K^{\textrm{top}}_{*,\R}(\Gamma;A)\longrightarrow K_{*,\R}^\Gamma(E\Gamma; A),
\end{equation}
given at the level of cycles by $t_{(X,m)}(\llbracket\, Y, y\,\rrbracket)= \llangle \,Y\times X, y\otimes [t_m]\,\rrangle$ {\rm(}where $Y$ is a proper $\Gamma$-compact space and $y\in KK^\Gamma_\R(C_0(Y),A)${\rm)}.
\item The composition $t_{(X,m)}\circ \sigma_*$ is the identity of $K_{*,\R}^\Gamma(E\Gamma; A)$.
\item The map $t_{(X,m)}$ does not depend on the choice of $(X,m)$. We will denote it by $t$.
\end{enumerate}
 
\end{prop}

\begin{proof} 
Let $\llbracket \, Y, y\, \rrbracket\in K_{*,\R}^{\textrm{top}}(\Gamma; A)$. 
Then $y \otimes [t_m]\in KK^\Gamma_{\R}(C_0(Y\times X),A)$ and $Y \times X$ maps to $E\Gamma$ because it is free and proper.

If $Z$ is another proper $\Gamma$-compact space and $f:Y\to Z$ is a $\Gamma$-equivariant map then $(f\times \operatorname{id}_X)_*(y\otimes [t_m])= f_*y\otimes [t_m]$. This shows that $t_{(X,m)}$ is well defined in the limit. 

If the action of $\Gamma$ on $Y$ is free and proper, using the map of free and proper $\Gamma$-spaces $q:Y\times X\to Y$, it follows that the class $ \llangle \,Y\times X, w\,\rrangle$ in $K_{*,\R}^\Gamma(E\Gamma; A)$ of an element $w\in KK_{*,\R}^\Gamma(C_0(Y\times X), A)$ is equal to $ \llangle \,Y, q_*(w)\,\rrangle$. In particular, $\llangle \,Y\times X, y\otimes [t_m]\,\rrangle=\llangle \,Y, y\otimes i_X^*[t_m]\,\rrangle=\llangle \,Y, y\otimes [\tau]\rrangle=\llangle \,Y, y\rrangle$. Thus (2) follows.

Given $(X_1,m_1)$ and $(X_2,m_2)$, put $X=X_1\times X_2$ and $m=m_1\times m_2$. Then, for every cycle $[Y,y]$, we have $\llangle Y\times X,y\otimes [t_m]\rrangle =t_{(X_2,m_2)}(\llbracket Y\times X_1,y\otimes [t_{m_1}]\rrbracket )$. Since the action of $\Gamma$ on $Y\times X_1$ is free and proper, it follows that $t_{(X_2,m_2)}(\llbracket Y\times X_1,y\otimes [t_{m_1}]\rrbracket )=\llangle Y\times X_1,y\otimes [t_{m_1}]\rrangle $ from (2), \emph{i.e.} $t_{(X,m)}(\llbracket Y,y\rrbracket)=t_{(X_1,m_1)}(\llbracket Y,y\rrbracket)$.

In the same way $t_{(X,m)}(\llbracket Y,y\rrbracket)=t_{(X_2,m_2)}(\llbracket Y,y\rrbracket)$.
\end{proof}

We will show below (Theorem \ref{tafsection}) that every discrete countable $\Gamma$ admits a compact (TAF) p.m.p.\ space. 

\medskip In the next theorem we are going to use this space.
Let $\sigma_{*}: K^{\Gamma}_{*,\R}(E\Gamma; A)\to K_{*,\R}^{\textrm{top}}(\Gamma; A)$ be induced by the natural map $\sigma: E\Gamma\to \underline E\Gamma$.

\begin{theorem} 
\label{th:sigmat}
The  morphism $t$ is a left inverse of $\sigma_*$. More precisely
\begin{enumerate}
\item $t\circ \sigma_*=\tau$ and $\sigma_{*}\circ t=\tau$, \textit{i.e}.\ the following diagram commutes 
$$
\xymatrix{
K_{*,\R}^{\G}(E\Gamma;A)\ar[r]^{\sigma_*}\ar[d]_{\tau=\id} & K_{*,\R}^{\operatorname{top}}(\Gamma;A)\ar[d]^\tau\ar[ld]^t\\
K_{*,\R}^{\G}(E\Gamma;A)\ar[r]_{\sigma_{*}}  & K_{*,\R}^{\operatorname{top}}(\Gamma;A)
}
 $$
\item As a consequence
we have an induced isomorphism
\begin{equation}
\label{eq:sigmaiso}
\sigma_*:\xymatrix{ K^\Gamma_{*,\R}(E\Gamma; A) \ar[r]^{{\scalebox{1.5}[1]{$\simeq\:\:$}}}&  K^{\textrm{top}}_{*,\R}(\Gamma;A)_{\tau}\ .
}
\end{equation}
\end{enumerate}
\end{theorem}

\begin{proof}
{We already proved in Prop. \ref{prop5.3}.2) that $t\circ\sigma_*$ is the identity of $K^\Gamma_{*,\R}(E\Gamma; A)$. \\
Let $Y$ be a free proper and cocompact $\Gamma$-space and $y\in KK^\Gamma(C_0(Y), A)$. Then \begin{align*}
(\sigma_* \circ t)(\llbracket  \, Y, y\, \rrbracket)&= \sigma_*(\llangle \,Y\times X, y\otimes [t_m]\,\rrangle)=  \llbracket Y\times X, y\otimes [t_m]\rrbracket\\& =\llbracket\, Y , (y\otimes  i_X(t_m))\,\rrbracket= [\tau]\llbracket \,Y, y\,\rrbracket
\end{align*}
 by Proposition \ref{prop:measure} (3).
{Notice that the third equality here follows from the defining property of $\underline{E}\Gamma$ as a direct limit over all the proper $\Gamma$-spaces.}
}
\end{proof}

\subsection*{A model of (TAF) space}
\label{sec:TAF}We construct a natural $(TAF)$ space associated to the discrete group $\Gamma$.
Let $F \subset \Gamma$ be a finite subgroup; define $X_F$ to be the space of all the (set-theoretical) sections of the projection $\pi:\Gamma \longrightarrow \Gamma/F$:
$$X_F:= \{s:\Gamma/F \longrightarrow \Gamma, \, \pi \circ s=\operatorname{Id}\}.$$
Notice we can write it as the direct product of all the cosets:
$$X_F=\prod_{[g] \in \Gamma/F}gF.
$$ 
We give $X_F$ the natural topological product structure. 
In particular it is a  compact (Cantor)  metrizable space.  Let us consider on every coset $gF$ the counting measure normalised by ${\#(F)}^{-1}$.
Then we define $m^F$ to be the product measure on  $X_F$. Thanks to the normalisation $m^F$ is a probability measure.

The group $\Gamma$ naturally acts on the left on $\Gamma$ and on $\Gamma/F$, then also on  $X_F$ by setting
$$\Gamma 
 \rotatebox[origin=c]{-90}{$\circlearrowright$}\ X_F
 , \quad  \quad (\gamma \cdot s)[x]=\gamma s([\gamma^{-1}x]), \quad s \in X_F,\, x \in \Gamma.$$
Indeed (since $\pi$ is equivariant), $\gamma \cdot s$ is a section. 
\begin{lemma}Properties of $X_F$:
\begin{enumerate}
	\item $F$ acts freely on $X_F$.
	\item The product measure $m^F$ is $\Gamma$-invariant.
\end{enumerate}
\end{lemma}
\begin{proof}
(1) If $a\in F$ fixes $s$, since $aF=F$, $s(F)=s(aF)=as(F)$ and thus $a=e$.

(2) By definition $m^F$ is the normalised counting measure on every cylindrical set. These cylindrical sets generate the Borel structure and are in the form $C_T$ where for a finite subset $T\subset \Gamma$ we put $$C_T=\{s\in X_F;\ s(\pi(x))=x,\ \hbox{for all}\ x\in T\}.$$
If $C_T\ne \emptyset$, \textit{i.e}.\ if $\pi $ is injective on $T$, we have $m^F(C_T)=\#(F)^{-\# (T)}$. As $g(C_T)=C_{g(T)}$, invariance of $m_F$ follows.
\end{proof}
   	
\begin{definition} We define
$$
 X\Gamma:=\prod_{ F \in \mathscr{F} } X_F,\quad \ \mathscr{F}:=\{ \textrm{finite subgroups } F \subset \G\},
$$
endowed with the product topology, with the product probability measure $m:=\prod m^F$ and, with the diagonal action of $\G$. \end{definition}
Then $X\Gamma$ is Hausdorff, compact and second-countable; $m$ is $\Gamma$-invariant and, by construction, every finite subgroup $F\subset \Gamma$ acts freely on one of the components - and thus on $X\Gamma$. 
We have  shown:
 \begin{theorem} 
 \label{tafsection}
 The compact (Cantor) space $ X\Gamma$ has property $(TAF)$ and $\Gamma$ preserves its probability measure.
\end{theorem}

\smallskip

\begin{remark}
It is also possible, although not needed in our construction, to construct a free compact $\Gamma$-space with a probability measure preserved.

Indeed, let $g\in \Gamma\setminus\{e\}$, let ${<}g{>}$ be the subgroup of $\Gamma $ it generates, and let $Y_g$ be a free compact ${<}g{>}$-space with a probability measure preserved. Let $X_g=(Y_g^\Gamma)^g$ be the subset of $Y_g^\Gamma$ of $f:\Gamma\to Y_g$ such that $f(hg)=g^{-1}f(h)$ for all $h\in\Gamma$. The space $X_g$ identifies with $Y_g^{\Gamma/{<}g{>}}$ thanks to any cross-section $s:\Gamma/{<}g{>}\to \Gamma$ and therefore carries a natural product probability measure (independent of this cross-section). Also, $\Gamma$ acts on $X_g$ by putting $(h.f)(k)=f(h^{-1}k)$ for every $h,k\in \Gamma$ and fixes this product measure.

Note that the map $f\mapsto f(e)$ is ${<}g{>}$-equivariant from $X_g$ to $Y_g$, and therefore $X_g$ has no $g$ fixed points. 

The product $\prod_g X_g$ is the desired space.
\end{remark}

\medskip

\section{Relation with the Novikov conjecture}
\label{sec:Novikov}
In this section we prove that the injectivity of $\mu_{\tau}$ implies the rational injectivity of the Mishchenko--Kasparov assembly map $\tilde \mu:K_*(B\Gamma) \longrightarrow K_*(C^*_r\Gamma) $ and a fortiori the Novikov conjecture.
Before we need a simple observation.
 \begin{remark}
 \label{isoreal}
We will use the fact that for $B\Gamma$, at the level of $K$-homology, adding real coefficients only discards torsion. In other words there is a canonical isomorphism
\begin{equation}
\label{eq:BGEG}
K^\G_{*,\R}(E\Gamma)\simeq   K_{*,\R}(B\Gamma)\simeq K_*(B\Gamma) \otimes \R\ .
\end{equation} 
This identification follows immediately starting from the definition of the compactly supported $K$-homology of $B\Gamma$ as a direct limit of $K$-homology groups of finite complexes $Z$. For each one of these compact pieces, the group $KK(C(Z),\C)$ is finitely generated and the analogous isomorphism $KK_{\R}(C(Z),\C)\simeq KK(C(Z),\C)\otimes \R$ holds. By the flatness of $\R$, the isomorphism is preserved by the direct limit.
\end{remark}

\bigskip

The Mishchenko--Kasparov assembly map will be denoted $\tilde \mu:K_*^\Gamma(E\Gamma)\longrightarrow K_*(C^*_r\Gamma)$. It is given by $\tilde \mu =\mu \circ \sigma_*$.  
The corresponding map with values in $K_*(C^*\Gamma)$ will be denoted $\tilde \mu_{\textrm{max}}: K_*^\Gamma(E\Gamma)\longrightarrow K_*(C^*\Gamma)$.

\begin{theorem}
\label{theorem:inj-novikov}
If $\mu_\tau:   K^{\textrm{top}}_{*, \R}(\Gamma)_{\tau}\longrightarrow K_\R(C^*_r \G)_{\tau}$ is injective then the Mishchenko--Kasparov assembly map $\tilde \mu : K_*^\Gamma(E\Gamma)\longrightarrow K_*(C^*_r\Gamma)$ is rationally injective.	
\end{theorem}
\begin{proof}  
It is enough to show that, under the assumption of injectivity of $\mu_\tau$, the kernel of $\tilde \mu$ consists only of torsion elements. 
Let $x\in K^{\Gamma}_{*}( E\Gamma)$ such that $\mu(\sigma_*(x))=0$ in $K_*(C_r^*\Gamma)$. 
By the definition of $\mu_\tau$, we deduce that the element  $\sigma_*(x)=\sigma_*(x)\otimes [\tau]$ is in the kernel of $\mu_\tau$, so that $\sigma_*(x)\otimes [\tau]=0$ in $ K_{*,\R}^{\textrm{top}}(\Gamma)_{\tau}$. 
Applying the map $t$ and using that by Theorem \ref{th:sigmat} $t\circ \sigma =\tau$,  we have 
$$
0=t(\sigma_*(x)\otimes [\tau])=t(\sigma_*(x))=x\otimes [\tau] \;\;\text{in }\; K_{*, \R}^\Gamma(E\Gamma)\ .
$$ 
By Remark \ref{isoreal}, we conclude that $x$ is torsion in $K_{*}^\Gamma(E\Gamma)$.
\end{proof}

\begin{remark} 
\label{diagram}
For every $C^*$-algebra $A$ there is a map $\beta: K_{*}(A)\otimes \mathbb R\longrightarrow K_{*,\mathbb R}(A)$ induced by the exterior Kasparov product.

The relation between  $\mu_\tau$ and $\tilde \mu\otimes 1$ used in the above proof is summarised in the following commutative diagram:
$$
	\xymatrix{&K^{\textrm{top}}_{*,\mathbb R}(\Gamma)_{\tau}\ar[r]^{\mu_\tau\;}\ar@{-->}@/^1.2pc/[dl]^t &K_{*,\mathbb R}(C_r^*\Gamma)_\tau\ar@{^{(}->}[d] \\
		{K^{\Gamma}_{*,\mathbb R}( E\Gamma)}\ar[ur]^{\sigma_*} 
&& K_{*,\mathbb R}(C_r^*\Gamma)\\
		{K^{\Gamma}_{*}( E\Gamma)}\otimes \mathbb R \ar[u]^{\cong}\ar[rr]_{\tilde \mu\otimes 1} && K_{*}(C_r^*\Gamma)\otimes \mathbb R\ar[u]_\beta\ .
	}
$$
The injectivity of the map $\sigma_{*}: K^{\Gamma}_{*,\mathbb{R}}(E\Gamma)\longrightarrow K_{*,\mathbb{R}}^{\textrm{top}}(\Gamma)$ (Theorem \ref{th:sigmat}) has recently been used in \cite{GWY} to show that the Novikov conjecture holds for any discrete group admitting an isometric and metrically proper action on an admissible Hilbert-Hadamard space.
\end{remark} 

\smallskip

\section{Exactness and monster-based counterexamples}
\label{sec:counterex}
We show that the construction of the counterexamples for group actions in \cite{HLS} still provides counterexample to the bijectivity of the $\tau$-Baum--Connes map $\mu_\tau$.

We know from Section \ref{functorial} that the $\tau$-parts are functorial. In particular an exact sequence
 $I \rightarrow A \rightarrow Q$ of $\G-C^*$-algebras induces a sequence 
\begin{equation}\label{sequencereducedtau}
\xymatrix{KK_{\R}(\C,I\rtimes_r \Gamma)_\tau\ar[r]& KK_{\R}(\C,A\rtimes_r \Gamma)_\tau \ar[r]&KK_{\R}(\C,Q\rtimes_r \Gamma)_\tau  }.
\end{equation}
We will show that exactly with the same Gromov's group $\Gamma$ and the same $C^*$-algebras $A$ and $I$ as in  \cite{HLS}, this is NOT exact. This is in spite of the fact that the $\tau$-parts for the maximal and the reduced algebras are the same, and that we could think of \eqref{sequencereducedtau} as induced by the full crossed product exact sequence $$I \rtimes \G \rightarrow A\rtimes \Gamma \rightarrow Q\rtimes \Gamma \ .$$  
In fact, we will show that taking the minimal tensor product with a $\rm{II}_1$-factor destroys exactness and this is why the sequence \ref{sequencereducedtau} fails exactness.

\bigskip We briefly recall the construction of \cite{HLS}.

Let  $\Gamma$ be a Gromov monster group (see \cite{AD, Gr} and the recent \cite{Cou, Osajda}). This is a finitely generated, discrete group, and one can map (in the sense of  Cayley graphs) an expanding sequence of finite graphs $X_n$ to $\Gamma$ in a controlled and \lq\lq essentially injective\rq\rq ~way.  Call $\varphi_n:X_n^0 \longrightarrow \Gamma$ such a collection of maps from the corresponding object spaces $X_n^0$.
 
 Given that, one can find a compact metrizable $\Gamma$-space $Z$ such that the Baum--Connes map with coefficients in $C(Z)\rtimes_r \Gamma$ is not an isomorphism. To do so, two passages are needed:
 \begin{enumerate}
 \item the construction of a sequence of $C^*$-algebras (involving non-separable ones) which is not exact in the middle. This is done by considering $\G$ acting by translation on the (non-separable) $C^*$-algebra $A:=\ell^{\infty}(\mathbb{N};c_0(\Gamma))$: then the sequence
  \begin{equation}
  \label{eq:seq-counter}
 \xymatrix{ c_0(\mathbb{N}\times \Gamma) \rtimes_r \Gamma \ar[r]& A \rtimes_r \Gamma \ar[r]& \big{(}A/c_0(\mathbb{N}\times \Gamma) \big{)}\rtimes_r \Gamma }	
 \end{equation}
is non exact in the middle. This happens even at level of $K$-theory. We will recall later why the  sequence:
 \begin{equation}
 \label{Knotexact}
 \xymatrix{ K_0(c_0(\mathbb{N}\times \Gamma) \rtimes_r \Gamma )\ar[r]& K_0(A \rtimes_r \Gamma \ar[r])&  K_0\big{(} \big{(}A/c_0(\mathbb{N}\times \Gamma) \big{)}\rtimes_r \Gamma    \big{)}}	\end{equation}
is not exact in the middle. 
 \item Building on $A$, an argument of direct limits and Gel'fand duality produces the compact $\G$-space $Z$. This is the separable counterexample. \end{enumerate}
The middle non exactness of \eqref{Knotexact} is shown directly: there is a class $[p]\in K_0(A\rtimes_r \Gamma)$ mapping to zero which cannot come from $K_0(c_0(\mathbb{N}\times \Gamma) \rtimes_r \Gamma)$.
The projection $$p \in A \rtimes_r \Gamma \subset \ell^{\infty}(\mathbb{N};\mathbb{K}(\ell^2(\Gamma))$$ is constructed using the spectral properties of the expander. First one represents faithfully $A\rtimes_r \Gamma$ on $\ell^2(\mathbb{N}\times \Gamma)$ then using the maps $\varphi_n$ and the graph Laplacians on each $X_n$ a bounded operator $D \in \widetilde{C_c(\Gamma;A)}$ (unitalization) is constructed. Zero is isolated in the spectrum of $D$ and $[p]$ is exactly the projection on the kernel. Seeing $p$ as a sequence of infinite matrices $p_n$, it is easy to check that
$$p_{n,y,z}=\dfrac{  \sqrt{ \# ( \varphi_n^{-1}(y) ) \#  (\varphi_n^{-1}(z))  }        }
{\#({X_n^0})}, \quad y,z \in \Gamma, \quad n \in \mathbb{N}.$$

By the essential injectivity of the maps $\varphi_n$, the coefficients of $p_n$ tend to $0$ when $n\to\infty$. It then follows that the image of $p\in \big{(}A/c_0(\mathbb{N}\times \Gamma) \big{)}\rtimes_r \Gamma $ is $0$.

\medskip The evaluation morphisms $\pi_n : \ell^\infty(\mathbb{N};c_0(\Gamma)) \rtimes_r \Gamma \longrightarrow  c_0(\Gamma)\rtimes \Gamma $ induce maps 
$$
\pi_n: K_0(A\rtimes_r \Gamma) \longrightarrow K_0(c_0(\Gamma)\rtimes_r \Gamma) \cong \mathbb{Z}
$$ 
such that $\pi_n([p])=[p_n]=1$ because $p_n$ is a rank one projection. It follows immediately that $[p]$ cannot be the image of an element in $K_0(c_0(\mathbb{N}\times \Gamma) \rtimes_r \Gamma )$ \eqref{Knotexact} because this is isomorphic to an algebraic direct sum $\oplus_{n\in \mathbb{N}}\mathbb{Z}$.

\medskip

Let us now pass to the $\tau$-parts:
\begin{theorem}
Let $\Gamma$, $A$ be as in \eqref{eq:seq-counter}. The sequence
 	\begin{equation}
 	\label{eq:tauseq}
 	 \xymatrix{ K_{0,\R}(c_0(\mathbb{N}\times \Gamma) \rtimes_r \Gamma )_{\tau}\ar[r]&  K_{0,\R}(A \rtimes_r \Gamma \ar[r])_{\tau}&   K_{0,\R}\big{(} \big{(}A/c_0(\mathbb{N}\times \Gamma) \big{)}\rtimes_r \Gamma    \big{)}_{\tau} }	 
 	 \end{equation} 
 fails to be exact in the middle.
\end{theorem}
\begin{proof}
Let $[p]_\R\in K_{0,\R}(A \rtimes_r \Gamma)$ be the image of the class  $[p]$ of the above projection via the change of coefficients. We show that the element $P:=[p]_\R\otimes [\tau^A]_r=[p]_\R\otimes J^\Gamma_r(1_A\otimes [\tau])$ in the middle group of the sequence is mapped to zero but does not come from the first group.

Since the image of $p$ vanishes already in $\big{(}A/c_0(\mathbb{N}\times \Gamma) \big{)}\rtimes_r \Gamma$ the first assertion follows immediately.

 Note that $K_{0, \R}(c_0(\Gamma)\rtimes_r \Gamma) \cong \mathbb{R}$: indeed $c_0(\Gamma)\rtimes_r \Gamma$ and $\C$ are Morita equivalent  (before tensoring with a $\rm{II}_1$-factor). In the same way, one sees that $K_{0, \R}(c_0(\mathbb{N}\times\Gamma)\rtimes_r \Gamma)$ is the algebraic direct sum $\oplus_{n\in \mathbb{N}}\mathbb{R}$. 

Now for every $n\in \mathbb{N}$, the morphism induced by the evaluation $\pi_n$ on the $\tau$-parts acts on $P$ as 
$$
\pi_n(P)=\pi_n([p]_\R\otimes [\tau^A]_r)=[p_n]_\R\otimes [\tau^{C_0(\Gamma)}]_r=[p_n]_{\R}$$
where we have applied that $[\tau]$ acts as the identity on the KFP algebra $C_0(\Gamma)$. In particular $\pi_n(P)$ is nonzero for every $n\in \mathbb{N}$. This implies that $P$
does not come from $K_{0,\R}(c_0(\mathbb{N}\times \Gamma)=\oplus_{n\in \mathbb{N}}\mathbb{R}.$
\end{proof}

\begin{prop}
\label{tauKtop-exact}
Let $A=\ell^\infty(\mathbb N, c_0(\G))$ and $I=c_0(\mathbb N\times \Gamma)$. Consider the exact sequence of  $\Gamma$-algebras
$I\to A\to A/I$. Then the corresponding sequence
$$
 \xymatrix{ K_{0,\R}^{\textrm{top}}(\Gamma, I)_{\tau}\ar[r]&  K_{0,\R}^{\textrm{top}}(\Gamma, A)_{\tau}\ar[r] &K_{0,\R}^{\textrm{top}}(\Gamma, A/I)_{\tau} }
$$
is exact.
\end{prop}

To show Proposition \ref{tauKtop-exact}, first recall that by \eqref{eq:sigmaiso} the isomorphism $K^{\textrm{top}}_{*,\R}(\Gamma;A)_{\tau}\simeq K^\Gamma_{*,\R}(E\Gamma; A)$ holds true and commutes with the morphisms induced by the sequence $I\to A\to A/I$, so that it is enough to show the following: 
\begin{lemma}
\label{lem2}
Let $Y$ be a free and proper, cocompact $\Gamma$-space. Let $I\to A\to A/I$ be the  exact sequence of  $\Gamma$-algebras appearing in \eqref{eq:seq-counter}. Then the sequence 
$$
 \xymatrix{KK^\Gamma_{\R}(C_0(Y), I)\ar[r]& KK^\Gamma_{\R}(C_0(Y), A)\ar[r] &KK^\Gamma_{\R}(C_0(Y), A/I)}
$$
is exact.
\end{lemma}

To show Lemma \ref{lem2}, let $N$ be any $\textrm{II}_1$-factor with trivial $\Gamma$-action. Nuclearity of $A$ and $I$ implies that the sequence \begin{equation}\label{tensoredsequence}
 	I\otimes N\to A\otimes N\to (A/I)\otimes N
 \end{equation}
 is still exact. A completely positive cross section of $A \longrightarrow A/I$ induces a completely positive cross section of \eqref{tensoredsequence} which thus remains semi-split.
 We now show the following fact in the slightly more general case of
any proper $\Gamma$-algebra $D$. In our application it will be $D=C_0(Y)$ which is free and proper.

\begin{lemma}
\label{lemma:exactness}
Let $D$ be a separable proper $\Gamma$-algebra. Let $\xymatrix{J\ar^j[r]&B\ar^q[r]& B/J}$
be  a  semi-split exact sequence of (possibly non-separable) $\Gamma$-$C^*$-algebras. Then the sequence
\begin{equation}
\label{claim}
\xymatrix{KK^\Gamma(D, J)\ar^{j_*}[r]& KK^\Gamma(D,  B)\ar^{q_*}[r] &KK^\Gamma(D,B/ J )}
\end{equation}
is exact. 
\end{lemma}
Once lemma \ref{lemma:exactness} is proved, we shall apply it to $D=C_0(Y)$ and to the exact sequence $I\otimes N\to A\otimes N\to (A/I)\otimes N$. The exactness of 
$$
\xymatrix{KK^\Gamma(C_0(Y), I\otimes N )\ar[r]& KK^\Gamma(C_0(Y),  A\otimes N )\ar[r] &KK^\Gamma(C_0(Y),(A/I)\otimes N  )}$$
will follow.
Finally, taking the limit over $N$ of the groups in the exact sequence above, we will prove Lemma \ref{lem2}. We are hence left with the
\begin{proof}{\emph{Lemma} \ref{lemma:exactness}: }
If $B$ and $J$ are separable, the statement is proved in \cite[Prop. 2.3]{EM}. Let us discuss the generalisation to the non-separable case, for which we follow closely \cite[proof of Prop. 3.1]{Sk}.
More precisely we assume that \eqref{claim} is middle exact in the separable case and we show that it is middle exact in full generality.

Let $x\in KK^\Gamma(D,B)$ such that $q_*(x)=0$ in $ KK^\Gamma(D, B/J)$. As $KK^\Gamma(D,B)$ is the inductive limit over separable subalgebras of $B$ (\emph{cf.} \cite{Sk2} and Lemma \ref{lem:sep}), there exists a separable $\Gamma$-invariant subalgebra $B_1$ of $B$ and an element $x_1\in KK^\Gamma(D,B_1)$ whose image by the inclusion $j_1:B_1\to B$ is $x$. Let $q_1:B_1\to q(B_1)$ be the restriction of $q$ to $B_1$ and $\ell_1:q(B_1)\to B/J$ be the inclusion. Since $q\circ j_1=\ell_1\circ q_1$, we find that $(\ell_1)_*((q_1)_*(x))=0$. 

As $KK^\Gamma(D,B/J)$ is the inductive limit over separable  $\Gamma$-invariant subalgebras of $B/J$, there exists a separable $\Gamma$-invariant subalgebra $Q_2\subset B/J$, containing $q(B_1)$, such that the image of $(q_1)_*(x)$ in $Q_2$ through the inclusion $\ell:q(B_1)\to Q_2$ is the $0$ element of $KK^\Gamma(D,Q_2)$. 
Let $s$ be a completely positive section of $q$ and let then $B_2$ be the (separable) $\Gamma$-invariant subalgebra of $B$ generated  by $s(Q_2)$, its $\Gamma $-translates, and $B_1$. Denote by $j:B_1\to B_2$  the inclusion. Note that $q(B_2)=Q_2$ and put $q_2:B_2\to Q_2$. We then have $(q_2)_*(j_*(x_1))=\ell_*((q_1)_*(x_1))=0$. 

Now set $x_2=j_*(x_1)$ and let $j_2:B_2\to B$ and $\ell_2:Q_2\to B/J$ be the inclusions.

Since $j_1=j_2\circ j$, we have $(j_2)_*(x_2)=(j_1)_*(x_1)$. 

Applying \cite[Prop. 2.3]{EM} to the exact sequence of separable $\Gamma$-algebras $$0\to J\cap B_2\to B_2\to Q_2\to 0,$$ 
we find that $x_2$ is the image of an element $y_2\in KK^\Gamma(D,J\cap B_2)$ whose image in $B_2$ is $x_2$ - whence its image in $B$ is $x$.
\end{proof}

The above argument provides a (non-separable) counterexample to the bijectivity of $\mu_\tau$, more precisely: 
\begin{prop}\label{HLSexample}
Let $A=\ell^\infty(\mathbb N, c_0(\G))$, $I=c_0(\mathbb N\times \Gamma)$ as before. If $\mu^{A/I}_\tau$ is injective, then $\mu^{A}_\tau$ is not surjective.
\end{prop}

\begin{proof}
Let $x\in K_{0,\R}(A\rtimes_r \Gamma)_{\tau}$ the element which fails exactness in the second line of the commutative diagram below. Suppose $\mu^{A/I}_\tau$ is injective and by contradiction that $\mu^{A}_\tau$ is surjective. Then  the preimage  $y\in K_{0,\R}^{\textrm{top}}(\Gamma, A)_{\tau} $ of $x$ is mapped to zero in $K_{0,\R}^{\textrm{top}}(\Gamma, A/I)$ by the injectiviy assumption. Since the second line is exact, $y$ comes from an element in $K_{0,\R}^{\textrm{top}}(\Gamma, I)_{\tau}$ which provides then the contradiction

$$
 \xymatrix{ 
K_{0,\R}^{\textrm{top}}(\Gamma, I)_{\tau}\ar[r]\ar[d]^{\mu^I_\tau}&  K_{0,\R}^{\textrm{top}}(\Gamma, A)_{\tau}\ar[d]^{\mu^A_\tau}\ar[r] &K_{0,\R}^{\textrm{top}}(\Gamma, A/I)_{\tau}\ar[d]^{\mu^{A/I}_\tau}\\ 
 K_{0,\R}(I\rtimes_r \Gamma)_{\tau}\ar[r]&  K_{0,\R}(A\rtimes_r \Gamma)_{\tau}\ar[r] &K_{0,\R}((A/I)\rtimes_r \Gamma)_{\tau} \ .
 }
$$
\end{proof}

\begin{remarks}
\begin{enumerate}
\item Using the argument in \cite[Remark 12]{HLS} based on double mapping cones, one may construct an abelian $C^*$-algebra $B$ such that $K^{top}_{\R,*}(\Gamma ;B)=0$ and $ K_{\R,*}(B\rtimes \Gamma)_{\tau}\ne 0$, whence $\mu_\tau$ is not surjective.

\item
We further show that we can pass to a separable counterexample.
Assume that for an algebra $B$ the map $\mu^{B}_\tau$ is not surjective. Then there exists a (commutative) separable $\Gamma$-subalgebra $D\subset B$ such that $\mu^{D}_\tau$ is not surjective.

Indeed, let $y\in K_{\R,*}(B\rtimes_r\Gamma)_{\tau}$.
By functoriality of the assembly $\mu_\tau$, it is enough to show that there exists a separable $\Gamma$-invariant subalgebra $B'\subset B$ and $y_1\in K_{\R,*}(B'\rtimes_r\Gamma)_{\tau}$ such that $y=i_*(y_1)$, where $i:B'\to B$ is the inclusion. In other words it is enough to show that 
\begin{equation} 
\label{eq:taulim}
K_{\R,*}(B\rtimes_r\Gamma)_{\tau}=\lim_{\substack{B' \subset B\\  B'\textrm{separable}\\ \Gamma-\textrm{invariant}}}K_{\R,*}(B'\rtimes_r\Gamma)_{\tau}\ .
\end{equation}
Let $y=[\tau]\otimes w$, for $w\in K_{\R,*}(B\rtimes_r\Gamma)$. Then there exists a ${\rm II}_1$-factor $N$ such that $w \in K_*((B\rtimes_r\Gamma)\otimes N)$. 
We apply now Lemma \ref{lem:sep}: there exists a separable $B'\subset B$ such that $w=(i\otimes 1)_*(z_1)$ for a $z_1\in K_*((B'\rtimes_r\Gamma)\otimes N)$, and where $i:B'\to B$ is the inclusion. Hence $y=[\tau]\otimes w=[\tau]\otimes (i\otimes 1)_*(z_1)=(i\otimes 1)_*([\tau]\otimes z_1)$ which shows \eqref{eq:taulim}. 

\item By Lafforgue's work \cite{LafforgueHyperbolic}, the Baum--Connes conjecture with coefficients for hyperbolic groups holds. Now, Gromov's monster is an inductive limit of hyperbolic groups. Despite the functoriality of our $\mu_\tau$, the failure of the bijectivity of $\mu_\tau$ of Proposition \ref{HLSexample} shows that the group $K_\R(A\rtimes_r\Gamma)_\tau $ is not compatible with inductive limits of groups. It may be worth noting that the difficulty comes from the fact that the group morphisms in this inductive limit are onto but not one-to-one. In other words, passing from one step to the next one we add relations, not generators.
\end{enumerate}
\end{remarks}

Let us finally comment on a possibility of fixing the non exactness of our construction along the lines of the  recent papers of Baum, Guentner and Willet \cite{BGW} and  of Buss, Echterhoff and Willett \cite{BEW}.

A modified assembly map involving a \emph{minimal exact  and Morita compatible} crossed product functor has been proposed in \cite{BGW}, as the ``right hand side'' that should be considered in the Baum--Connes conjecture (see also \cite{ABES} for some relations with the strong Novikov conjecture for low degree classes).

\medskip
 
One may try to correct the non-exactness in our context by correcting the right hand side for the $KK$-theory with real coefficients. One could  for instance think of replacing the minimal tensor product with ${\rm II}_1$-factors by a tensor product with ${\rm II}_1$-factors which is \emph{minimal exact}  (and Morita compatible).

\noindent Paolo Antonini\\ 
Scuola Internazionale Superiore di Studi Avanzati\\
via Bonomea, 265\\
34136 Trieste, Italy\\
\texttt{pantonin@sissa.it}
\bigskip
\smallskip

\noindent Sara Azzali\\
Fachbereich Mathematik\\
Universit\"at Hamburg\\
Bundesstrasse 55\\
20146 Hamburg, Germany\\
\texttt{sara.azzali@uni-hamburg.de}

\bigskip
\smallskip

\noindent Georges Skandalis\\
Universit\'e Paris Diderot, Sorbonne Paris Cit\'e\\
Sorbonne Universit\'es, UPMC Paris 06, CNRS, IMJ-PRG\\
UFR de Math\'ematiques, {\sc CP} {\bf 7012} - B\^atiment Sophie Germain \\
5 rue Thomas Mann, 75205 Paris CEDEX 13, France\\
\texttt{skandalis@math.univ-paris-diderot.fr}

\bigskip
\bigskip

\end{document}